\documentclass[10pt]{article}

\usepackage{amssymb, amscd, latexsym}
\usepackage{amsmath}
\usepackage{amsthm}
\usepackage[dvips]{graphicx}
\usepackage[all]{xy}
\usepackage{mathrsfs}
\usepackage{pstricks-add}
\usepackage[utf8]{inputenc}
\usepackage{multirow}
\usepackage{array}
\usepackage[colorlinks,linkcolor=blue,citecolor=magenta,anchorcolor=red,bookmarksopen,urlcolor=red,filecolor=red,
,pdfpagetransition={Wipe}]{hyperref}
\hypersetup{pdftitle={Yazdan Pour}}

\newtheorem{thm}{Theorem}[section]
\newtheorem{cor}[thm]{Corollary}
\newtheorem{lem}[thm]{Lemma}
\newtheorem{prop}[thm]{Proposition}
\newtheorem{defn}[thm]{Definition}
\newtheorem{rem}[thm]{Remark}
\newtheorem{ex}[thm]{Example}
\newtheorem{conj}[thm]{Conjecture}

\newcommand{\subject}[1]{\begin{flushleft}\textbf{Mathematics  Subject Classification[2010]}: #1\end{flushleft}}
\newcommand{\keyword}[1]{\par\noindent \textbf{Keywords:} #1 }

\def\cocoa{{\hbox{\rm C\kern-.13em o\kern-.07em C\kern-.13em o\kern-.15em A}}}

\def\sqr#1#2{{\vcenter{\hrule height.#2pt
        \hbox{\vrule width.#2pt height#1pt \kern#1pt
            \vrule width.#2pt}
        \hrule height.#2pt}}}

\def\depth{{\rm depth}\,}

\def\ker{{\rm Ker}\,}
\def\reg{{\rm reg}\,}

\def\projdim{{\rm proj dim}\,}

\def\NN{{\mathbb N}}

\def\ZZ{{\mathbb Z}}

\def\C{{\mathcal C}}

\begin{document}

\title{\bf Regularity and Free Resolution of Ideals which are Minimal to $d$-linearity}
\author{Marcel Morales$^{1,3}$, Ali Akbar Yazdan Pour$^{1,2}$, Rashid Zaare-Nahandi$^2$\\
\small $^{1}$ Universit\'e de Grenoble I, Institut Fourier, Laboratoire de Math\'ematiques, France\\
\small $^{2}$ Institute for Advanced Studies in Basic Sciences, P. O. Box 45195-1159, Zanjan, Iran\\
\small $^{3}$ IUFM, Universit\'e de Lyon I, France}

\date{}
\maketitle

\begin{abstract}

Toward a partial classification of monomial ideals with $d$-linear resolution, in this paper,  some classes of $d$-uniform clutters which do not have linear resolution, but every proper subclutter of them has a $d$-linear resolution, are introduced and the regularity and Betti numbers of circuit ideals of such clutters are computed. Also, it is proved that for given two $d$-uniform clutters $\mathcal{C}_1, \mathcal{C}_2$, the Castelnuovo-Mumford regularity of the ideal $I(\overline{\mathcal{C}_1 \cup \mathcal{C}_2})$ is equal to the maximum of regularities of $I(\bar{\C}_1)$ and $I(\bar{\C}_2)$, whenever $V(\mathcal{C}_1) \cap V(\mathcal{C}_2)$ is a clique or ${\rm SC}(\mathcal{C}_1) \cap {\rm SC}(\mathcal{C}_2)=\emptyset$.

As applications, alternative proofs are given for Fr\"oberg's Theorem on linearity of edge ideal of graphs with chordal complement as well as for linearity of generalized chordal hypergraphs defined by Emtander. Finally, we find minimal free resolutions of the circuit ideal of a triangulation of a pseudo-manifold and a homology manifold explicitly.
\medskip
\keyword{minimal free resolution, Castelnuovo-Mumford regularity, clutter, Betti number, pseudo-manifold, triangulation.}
\subject{13D02, 51H30}
\end{abstract}

\section{Introduction}\label{Section1}

Although the problem of classification of monomial ideals with $d$-linear resolution is solved for $d=2$, it is still open for $d>2$. Passing via polarization, it is enough to solve the problem for square-free monomial ideals. An ideal generated by square-free monomials of degree $2$ can be assumed as edge ideal of a graph and more generally, an ideal generated by square-free monomials  of degree $d$ is the circuit ideal of a $d$-uniform clutter. R. Fr\"oberg~\cite{Fr} proved that the edge ideal of a graph $G$ has a $2$-linear resolution if and only if in the complement graph of $G$ every cycle of length greater than $3$ has a chord. In this case, linearity of resolution is not depending to characteristics of the ground field. To generalize the Fr\"oberg's result to higher dimensional clutters, we face that linearity of resolution of a circuit ideal of a $d$-uniform clutter for $d>2$ depends on the characteristics of the ground field. For instance, the ideal corresponding to triangulation of the projective plane has a linear resolution in characteristics 0 while it does not have linear resolution in characteristics 2. In a new proof of Fr\"oberg's Theorem in \cite{mar}, the notion of cycle plays a key role. That means:
\begin{itemize}
\item[\rm (1)] Cycles are exactly those graphs that are minimal to 2-linearity.
\item[\rm (2)] The edge ideal of $\bar G$ does not have 2-linear resolution if and only if $G$ contains a cycle of length $> 3$, as induced subgraph.
\end{itemize}

Trying to find a similar notion for cycles, we introduce the notion of minimal to $d$-linearity in arbitrary $d$-uniform  clutters. By Proposition~\ref{Pseudo-manifolds are minimal to linearity}, pseudo-manifolds have the property of minimal to $d$-linearity. Also we know that, if $\C$ is a $d$-uniform clutter which has an induced subclutter isomorphic to a $d$-dimensional pseudo-manifold, then the ideal $I(\bar{\C})$ does not have linear resolution. But, Example~\ref{non-pseudo-manifold}, shows that the class of pseudo-manifolds is strictly contained in the class of minimal to linearity clutters.   
Another difficulty for generalizing the Fr\"oberg's Theorem, is the term `induced' in (2). That is, there are clutters which do not have a linear resolution and do not have any induced subclutter minimal to $d$-linearity. For instance,  consider $\C$ is a triangulation of the sphere (with large enough number of vertices), which is a pseudo-manifold, let $v_1, v_2, v_3$ be vertices such that $v_1,v_2$ belong to a circuit of $\C$ and neither $v_1, v_3$ nor $v_2,v_3$ belong to any circuit. Then add a new circuit $\{v_1,v_2,v_3\}$ to $\C$. The new clutter does not have any induced subclutter which is minimal to $d$-linearity, however its circuit ideal does not have $d$-linear resolution.   

In \cite{Emtander, HaVanTuyl, VanTuylVillarreal, Wood1} the authors have partially generalized the Fr\"oberg's Theorem. They have introduced several definitions of chordal clutters and proved that corresponding circuit ideals have linear resolution. In~\cite{maar}, the notion of simplicial submaximal circuit is introduced and proved that removing such submaximal circuits does not change the regularity of the circuit ideal. This proves linearity of resolution of a large class of clutters (Remark 3.10 in \cite{maar}). To attack to this problem from another side, in the present paper, we investigate clutters which does not have linear resolution, but any proper subclutter of them has a linear resolution.

Section~\ref{Section2} is devoted to collect prerequisites and basic definitions which we need in the next chapters. In Section~\ref{Section3}, some homological behaviours  of the Stanley-Reisner ideal of a simplicial complex $\Delta$ with ${\rm indeg}\, (I_\Delta) \geq 1+ \dim \Delta$ are investigated and some minor extendings are made for results of Terai and Yoshida in \cite{Terai2}. 

Sections~\ref{Section4} and \ref{Section5} contain main results of this paper. Section~\ref{Section4} is about uniform clutters and their circuit ideals. In this section, we prove that for two $d$-uniform clutters $\mathcal{C}_1, \mathcal{C}_2$, the Castelnuovo-Mumford regularity of the ideal $I(\overline{\mathcal{C}_1 \cup \mathcal{C}_2})$, is maximum of the regularities of these two components, whenever $V(\mathcal{C}_1) \cap V(\mathcal{C}_2)$ is a clique or ${\rm SC}(\mathcal{C}_1) \cap {\rm SC}(\mathcal{C}_2)=\varnothing$ (See Definition~\ref{SC}).  In Section~\ref{Section5}, we define notions of obstruction to $d$-linearity, minimal to $d$-linearity and almost tree clutters. These are clutters such that their circuit ideals do not have $d$-linear resolution but any proper subclutter of them has a $d$-linear resolution. We compare these classes and then, compute explicitly the minimal free resolution of clutters which are minimal to $d$-linearity.

In Section~\ref{Section6}, as some applications to the results of previous sections, we give an alternative proof for the Fr\"oberg's theorem. Also a proof for linearity of resolution of generalized chordal hypergraphs defined by Emtander in~\cite{Emtander} is given. Finally, we find minimal free resolutions of circuit ideals of triangulations of pseudo-manifolds and homology manifolds.

\section{Preliminaries}\label{Section2}

Let $K$ be a field and $(R,{\mathfrak m})$ a Noetherian graded local ring with residue field
$K$. Let $M$ be a finitely generated graded $R$-module and
$$
 \cdots \to F_2 \to F_1 \to F_0 \to M \to 0
$$
a minimal graded free resolution of $M$ with $F_i = \oplus_j R(-j)^{\beta^K_{i,j}}$ for all $i$.

The numbers $\beta_{i,j}^K(M) = \dim_K \mbox{Tor}^R_i(K,M)_j$ are called the \textit{graded Betti numbers}
of $M$ and
$$
\mbox{projdim}(M) = \sup\{i : \quad \mbox{Tor}^R_i(K,M) \neq 0\}
$$
is called the projective dimension of $M$. Throughout this paper, we fix a field $K$ and for convenience we write simply $\beta_{i,j}$ instead of $\beta_{i,j}^K$. The Auslander-Buchsbaum Theorem enables us to find the projective dimension in terms of depth.

\begin{thm}[{Auslander-Buchsbaum \cite[Theorem 1.3.3]{BH}}] \label{AuslanderBuchsbaum}
Let $(R,\mathfrak{m})$ be a Noetherian local ring, and $M \neq 0$ a finitely generated $R$-module. If $\projdim M < \infty $, then
$$\projdim M + \depth M = \depth R.$$
\end{thm}

The \textit{Castelnuovo-Mumford regularity} $\mbox{reg}(M)$ of $M\neq 0$ is given by
$$
\mbox{reg}(M) = \sup\{j - i : \quad \beta_{i,j}(M) \neq 0\}.
$$
The \textit{initial degree} $\mbox{indeg}(M)$ of $M$ is given by
$$
\mbox{indeg}(M) = \inf\{i : \quad M_i \neq 0\}.
$$
We say that a finitely generated graded $R$-module $M$ has a\textit{ $d$-linear resolution} if its
regularity is equal to $d = \mbox{indeg}(M)$.

A \textit{simplicial complex} $\Delta$  over a set of vertices $V=\{ v_{1}, \ldots, v_{n} \}$ is a collection of subsets of $V$, such that $\{ v_{i} \} \in \Delta $ for all $i$, and if $F\in \Delta$, then all subsets of $F$ are also in $\Delta$ (including the empty set). An element of $\Delta$ is called a \textit{face} of $\Delta$, and the \textit{dimension} of a face $F$ of $\Delta$ is $\left|F\right|-1$, where $\left|F\right|$ is the number of elements of $F$. The maximal faces of $\Delta$ under inclusion are called \textit{facets} of $\Delta$. The \textit{dimension} of  $\Delta$, $\dim \Delta$, is the maximum of dimensions of its facets. Let $\mathcal{F}(\Delta) =\{F_{1}, \ldots, F_{q}\}$ be the facet set of $\Delta$. A simplicial complex $\Gamma$ is called a \textit{subcomplex} of $\Delta$ if $\mathcal{F}(\Gamma) \subset \mathcal{F}(\Delta)$. The \emph{non-face ideal} or the \emph{Stanley-Reisner ideal} of $\Delta$, denoted by $I_\Delta$, is the ideal of $S=K[x_1, \ldots, x_n]$ generated by square-free monomials $\{x_{i_1} \cdots x_{i_r} | \{ v_{i_1}, \ldots, v_{i_r}\} \notin \Delta \}$. Also we call $K[\Delta]:=S/I_\Delta$ the \emph{Stanley-Reisner ring} of $\Delta$. We have
$$I_\Delta=\bigcap\limits_{F \in \mathcal{F}(\Delta)}P_{\bar{F}}$$
where $P_{\bar{F}}$ denotes the (prime) ideal generated by $\{x_i | v_i \notin F \}$. In particular,
$\dim K[\Delta] =1+ \dim \Delta$.

For a simplicial complex $\Delta$ of dimension $d$, let $f_i = f_i(\Delta) $ denote the number of faces of $\Delta$ of dimension $i$ and by convention  $f_{-1}=1$. The sequence $\textbf{f}(\Delta) = (f_{-1}, f_0, \ldots, f_{d-1})$ is called the $\textbf{f}$-\textit{vector} of $\Delta$.

Let $\Delta$ be a simplicial complex with vertex set $V$. An \textit{orientation} on $\Delta$ is a linear order on $V$. A simplicial complex together with an orientation is an \emph{oriented simplicial complex}.

Suppose $\Delta$ is an oriented simplicial complex of dimension $d$, and $F \in \Delta$ a face of dimension $i$. We write $F = [v_0, \ldots, v_i]$ if $F = \{v_0, \ldots, v_i\}$ and $v_0< \cdots< v_i$, and $F = [\ ]$ if $F = \varnothing$. With this notation, we define the \emph{augmented oriented chain complex of $\Delta$},

$$\begin{CD}
\tilde{\mathscr{C}}(\Delta):\quad 0 @>\partial_{d+1}>> \C_{d} @>\partial_d>> \C_{d-1} @>\partial_{d-1}>> \cdots  @>\partial_{1}>> \C_{0} @>\partial_{0}>> \C_{-1} \longrightarrow 0
\end{CD}$$
by setting
\begin{center}
$\C_i=\bigoplus\limits_{\substack{F \in \Delta  \\ \dim F = i}}{KF} \qquad \text{ and } \qquad \partial_i(F)=\sum \limits_{j=1}^{i}{{(-1)}^jF_j}$
\end{center}
for all $F \in \Delta$; here $F_j = [v_0, \ldots, \hat{v}_j,\ldots, v_i]$ for $F = [v_0, \ldots, v_i]$. A straightforward computation shows that $\partial_i \circ \partial_{i+1}=0$. We set
$$\tilde{H}_i(\Delta; K)=H_i(\tilde{\mathscr{C}}(\Delta))=\frac{\ker \partial_i}{{\rm Im}\, \partial_{i+1}}, \qquad i = -1 ,\ldots, d,$$
and call $\tilde{H}_i(\Delta; K)$ the \emph{$i$-th reduced simplicial homology of $\Delta$}.

If $\Delta$ is a simplicial complex and $\Delta_1$ and $\Delta_2$ are subcomplexes of $\Delta$, then there is an exact sequence
\begin{align} \label{Reduced Mayer-Vietoris sequence}
\cdots \to \tilde{H}_j(\Delta_1 \cap \Delta_2 ; K) \to \tilde{H}_j(\Delta_1; K) \oplus \tilde{H}_j(\Delta_2; K) \to  \tilde{H}_j(\Delta_1 \cup \Delta_2 ; K) \to \tilde{H}_{j-1}(\Delta_1 \cap \Delta_2; K) \to \cdots
\end{align}
with all coefficients in $K$ called the \textit{reduced Mayer-Vietoris sequence} of $\Delta_1$ and $\Delta_2$.

\begin{defn}\rm
Let $\Delta$ be a simplicial complex of dimension $d$ and $\textbf{f}(\Delta) =(f_{-1},f_0, \ldots, f_d)$ be the $\textbf{f}$-vector of $\Delta$. The number
$$\chi(\Delta)= \sum\limits_{i=0}^{d} (-1)^if_i$$
is called the \textit{Euler characteristic} of $\Delta$.

In terms of simplicial homology, one has
\begin{equation}\label{Euler characteristic in terms of Homology}
-1+ \chi(\Delta)= \sum\limits_{i=0}^{d} (-1)^i \dim_K \tilde{H}_i(\Delta;K).
\end{equation}
\end{defn}

Hochster's formula describes the Betti number of a square-free monomial ideal $I$ in terms of the dimension of reduced homology of $\Delta$, when $I=I_\Delta$.

\begin{thm}[Hochster formula] \label{Hochster Formula}
Let $\Delta$ be a simplicial complex on $[n]$. Then,
\begin{equation*}
\beta^K_{i,j}(I_\Delta) = \sum\limits_{\substack{W \subset [n] \\ |W|=j}}{\dim_K \tilde{H}_{j-i-2}(\Delta_W; K)},
\end{equation*}
where $\Delta_W$ is the simplicial complex with vertex set $W$ and all faces of $\Delta$ with vertices in $W$.
\end{thm}

The following theorem, extends the well-known Herzog-Kuhl equations~\cite{HerzogKuhl} in the case of $\beta_{i, d_{i+1}} (M)=0$ for all $i \leq 0$.

\begin{thm}[\cite{D}] \label{dif-herzog-kuhl}
Let be $M$ a $\NN$-graded $S$-module, $\rho$ its projective dimension and $\text{{\rm \textbf{d}}} = ( d_0 < d_1 < \dots < d_\rho < d_{\rho+1} ) \in \NN^{\rho+2}$, such that $M$ has a free resolution with the following form:
\begin{equation*}
\begin{split}
0 & \rightarrow S(-d_{\rho+1})^{\beta_{\rho,d_{\rho+1}}} \oplus S(-d_{\rho})^{\beta_{\rho,d_{\rho}}} \rightarrow S(-d_{\rho})^{\beta_{\rho-1,d_\rho}}\oplus S(-d_{\rho-1})^{\beta_{\rho-1,d_{\rho-1}}}\rightarrow \\
& \rightarrow \dots \rightarrow S(-d_{2})^{\beta_{2,d_2}}\oplus S(-d_{1})^{\beta_{1,d_{1}}}\rightarrow S(-d_{1})^{\beta_{0,d_{1}}}\oplus S(-d_{0})^{\beta_{0,d_{0}}} \rightarrow M \rightarrow 0.
\end{split}
\end{equation*}

For $1\leq i\leq \rho$, let $\beta_i'=\beta_{i,d_i}-\beta_{i-1,d_i}$. Then we have:
\begin{itemize}
 \item[\rm (i)] If ${\rm depth}(M)={\rm dim}\mbox{ }M\mbox{ }{\rm and}\mbox{ }\beta_{\rho,d_{\rho +1}}=0$, then for all $1\leq i\leq \rho$,
    $$\beta_i'=\beta_0(-1)^{i}\prod\limits_{k = 1 \atop k \neq i}^{\rho} \left( \frac{d_k-d_0}{d_k-d_i} \right).$$
 \item[\rm (ii)] If ${\rm depth}(M)={\rm dim}\mbox{ }M,\mbox{ }\mbox{ }\beta_{\rho,d_{\rho +1}}\neq 0\mbox{ }{\rm and}\mbox{ }d_0=0$, then for all $1\leq i\leq \rho+1$,
$$\beta_i'=(-1)^{i-1}\frac{{\beta_0} \left( \prod\limits_{k=1 \atop k\neq i}^{\rho+1}d_k \right)- (\rho!) e(M)}{\prod\limits_{k=1 \atop k\neq i}^{\rho+1}(d_k-d_i)}.$$
 \item[\rm (iii)] If ${\rm depth}(M)={\rm dim}\mbox{ }M-1,\mbox{ }\beta_{\rho,d_{\rho +1}}=0\mbox{ }{\rm and}\mbox{ }d_0=0$, then for all $1\leq i\leq \rho$,
$$\beta_i'=(-1)^{i-1}\frac{{\beta_0} \left( \prod\limits_{k=1 \atop k\neq i}^{\rho}d_k \right)-(\rho-1)!e(M)}{\prod\limits_{k=1 \atop k\neq i}^{\rho}(d_k-d_i)}.$$
\end{itemize}
\end{thm}

\section{Simplicial Complexes $\Delta$ with ${\rm indeg}\, (I_\Delta) \geq 1+ \dim \Delta$}\label{Section3}

As we shall see later, the ideals which are minimal to linearity are located in the class of square-free monomial ideals $I_\Delta$, with ${\rm indeg}\, (I_\Delta) = 1+ \dim \Delta$ (see Definition~\ref{Definition of minimal to linearity}). For square-free monomial ideal $I$ with ${\rm indeg}\, (I) \geq d$, we have the following proposition.

\begin{prop} \label{Zero Reduced Homology of indeg=1+dim}
Let $\Delta$ be a simplicial complex on $[n]$ and $d$ be an integer such that ${\rm indeg}\, (I_\Delta) \geq d$. Then,
\begin{itemize}
\item[\rm (i)] $\tilde{H}_i(\Delta_W; K) =0$, for all $i<d-2$ and $W \subset [n]$.
\item[\rm (ii)] If $\beta_{i,j}(I_\Delta) \neq 0$, then, $1 \leq j \leq n$ and $d \leq j-i \leq \dim \Delta + 2$.
\end{itemize}
\end{prop}

\begin{proof}
(i) Let $\dim \Delta=r$ and
\begin{align*}
\begin{CD}
\tilde{\mathscr{C}}(\Delta):\quad 0 @>>>\C_r @>\partial_{r}>> \cdots @>\partial_{d+1}>> \C_d@>\partial_{d}>> \C_{d-1} \\ @>\partial_{d-1}>> \C_{d-2} @>\partial_{d-2}>> \cdots
@>\partial_{1}>> \C_{0} @>\partial_{0}>> \C_{-1} @>>> 0
\end{CD}
\end{align*}
be the augmented chain complex of $\Delta$ and $\Delta^{(d-2)}$ be $(d-2)$-skeleton of $\Delta$. That is $\Delta^{(d-2)} = \{ F \in \Delta, \, \, \dim F \leq d-2 \}$. Then the augmented chain complex of $\Delta^{(d-2)}$ is:
$$\begin{CD}
\tilde{\mathscr{C}}(\Delta^{(d-2)}):\quad  0 \to \C_{d-2} @>{\partial_{d-2}}>>  \cdots @>>> \C_1 @>{\partial_{1}}>> \C_0 @>{\partial_{0}}>> \C_{-1} \longrightarrow 0.
\end{CD}$$
So that $\tilde{H}_i(\Delta;K)=\tilde{H}_i(\Delta^{(d-2)};K)$ for $i < d-2$. Since, ${\rm indeg}\, (I_\Delta) \geq d$, the facet set of the complex $\Delta^{(d-2)}$ is all $(d-1)$-subsets of $[n]$. Hence $\tilde{H}_i(\Delta;K)=\tilde{H}_i(\Delta^{(d-2)};K)=0$ for $i < d-2$.

Moreover, if $W \subset [n]$, then all $(d-1)$-subsets of $W$ is again in $\Delta_W$. This implies that ${\rm indeg}\, (I_{\Delta_W}) \geq d$. Hence by what we have already proved, we conclude that $\tilde{H}_i(\Delta_W; K) =0$ for all $i< d-2$. This completes the proof.

(ii) If $\beta_{i,j}(I_\Delta) \neq 0$, then by Theorem~\ref{Hochster Formula}, there exists $W \subset [n]$ with $|W|=j$ and $\tilde{H}_{j-i-2}(\Delta_W; K) \neq 0$. So that, $1 \leq j=|W| \leq n$ and $j-i-2 \leq \dim \Delta$. Moreover,  by  part (i), we have $j-i-2 \geq d-2$.
\end{proof}

\begin{rem} \rm
Let $\Delta$ be a $(d-1)$-dimensional simplicial comlex such that ${\rm indeg}\, (I_\Delta) \geq d$. The main property of $\Delta$ is that it contains all faces of dimension $d-2$. Hence $\Delta$ contains all faces of dimension $-1,0,\ldots, d-2$. So that
\begin{equation}\label{f-Vector with indeg=1+dim}
f_i={n \choose i+1}, \quad i=-1, \ldots, d-2.
\end{equation}
\end{rem}
For a monomial ideal $I$, let $\mu (I)$ denotes the number of the minimal generators of $I$ and $e(I)$ denotes the multiplicity of $I$. As a consequence of Proposition~\ref{Zero Reduced Homology of indeg=1+dim}, we have:

\begin{cor} \label{Homology of ideals with indeg I = 1+dim}
Let $\Delta$ be a $(d-1)$-dimensional simplicial complex on $[n]$ such that ${\rm indeg} \,(I_\Delta) \geq d$. Then,
\begin{align}
\dim_K \tilde{H}_{d-2}(\Delta;K) - \dim_K \tilde{H}_{d-1}(\Delta;K) = \sum\limits_{i=0}^{d-1} (-1)^{d+i-1}{n \choose i} - e(S/I_\Delta).
\end{align}
\end{cor}

\begin{proof}
Using (\ref{Euler characteristic in terms of Homology}), Proposition~\ref{Zero Reduced Homology of indeg=1+dim} and (\ref{f-Vector with indeg=1+dim}), we have:
\begin{equation*}
(-1)^{d-2} \dim_K \tilde{H}_{d-2}(\Delta;K) + (-1)^{d-1} \dim_K \tilde{H}_{d-1}(\Delta;K) = -1 + (-1)^{d-1}f_{d-1}
+ \sum\limits_{i=0}^{d-2} (-1)^i{n \choose i+1}.
\end{equation*}
Since $e(S/I_\Delta)=f_{d-1}$, we get the conclusion.
\end{proof}

The following theorems, extend some results of Terai and Yoshida (c.f. {\cite{Terai2}}).


\begin{thm}\label{betti, reg and C-M with indeg=1+dim}
Let $\Delta$ be a $(d-1)$-dimensional simplicial complex on $[n]$ such that ${\rm indeg}\, (I_\Delta) \geq d$. Then,
\begin{itemize}
\item[\rm (i)] If $\beta_{i,j} (I_\Delta) \neq 0$, then $1 \leq j \leq n$ and $d \leq j-i \leq d+1$.
\item[\rm (ii)] $d \leq \reg (I_\Delta) \leq d+1.$
\item[\rm (iii)] ${\rm indeg}\, I_\Delta \leq d+1$ and equality holds if and only if  $I_\Delta$ has $(d+1)$-linear resolution.
\item[\rm (iv)] $(n-d)-1 \leq \projdim (I_\Delta) \leq n-d$.
\end{itemize}
\end{thm}

\begin{proof}
(i) If $\beta_{i,j} (I_\Delta) \neq 0$, then by Theorem~\ref{Hochster Formula}, there exists $\emptyset \neq W \subset [n]$, such that $|W|=j$ and $\tilde{H}_{j-i-2}(\Delta_W; K) \neq 0$. So that $1 \leq j \leq n$ and by Proposition~\ref{Zero Reduced Homology of indeg=1+dim}, $d-2 \leq j-i-2 \leq d-1$. That is, $d \leq j-i \leq d+1$.

(ii) By part (i), we have
\begin{equation*}
d \leq  {\rm indeg}\,(I_\Delta) \leq \reg (I_\Delta) = \max \{j-i \colon \quad \beta_{i,j} \neq 0\} \leq d+1.
\end{equation*}

(iii) If $x_{i_1}\cdots x_{i_j} \in I_\Delta$, then $\beta_{0,j} \neq 0$. So that by (i), $j \leq d+1$. In particular, ${\rm indeg}\,(I_\Delta) \leq d+1$.

If ${\rm indeg}\,(I_\Delta) = d+1$, then $ \reg(I_\Delta) \geq d+1$ and by (ii), $I_\Delta$ has $(d+1)$-linear resolution. On the other hand, if $I_\Delta$ has $(d+1)$-linear resolution, then each generator has degree $d+1$. So that ${\rm indeg}\,(I_\Delta) = d+1$.

(iv) Let $\rho =\projdim (I_\Delta)$. By Theorem~\ref{AuslanderBuchsbaum},
\begin{equation*}
\rho + 1 = \projdim \frac{S}{I_\Delta} = n- \depth \frac{S}{I_\Delta} \geq n-\dim \frac{S}{I_\Delta} = n-d.
\end{equation*}
Hence $\rho \geq (n-d)-1$.

On the other hand, $\beta_\rho (I_\Delta) \neq 0$. Hence, there exists $1 \leq j \leq n$, such that $\beta_{\rho,j} \neq 0$. So, by (i), $j-\rho \geq d$. This implies that $\rho \leq j-d \leq n-d$.
\end{proof}

\begin{thm} \label{Cohen_macaulayness with indeg=1+dim}
Let $S=K[x_1, \ldots, x_n]$ be the polynomial ring over a field $K$ and $\Delta$ be a $(d-1)$-dimensional simplicial complex on $[n]$ such that ${\rm indeg}\, (I_\Delta) \geq d$. Then, $S/I_\Delta$ is Cohen-Macaulay if and only if  $\tilde{H}_{d-2}(\Delta; K) =0$.
\end{thm}

\begin{proof}
We know that $\dim S/I_\Delta=d$. So that Theorem~\ref{AuslanderBuchsbaum}, implies that
\begin{center}
$S/I_\Delta$ is Cohen-Macaulay if and only if $\projdim S/I_\Delta= (n-d)$.
\end{center}
In view of Theorem~\ref{betti, reg and C-M with indeg=1+dim}(iv), it is enough to prove that
\begin{equation*}
\projdim S/I_\Delta= (n-d)+1 \Longleftrightarrow \tilde{H}_{d-2}(\Delta; K) \neq 0.
\end{equation*}

($\Leftarrow$) If $\tilde{H}_{d-2}(\Delta; K) \neq 0$, then by Theorem~\ref{Hochster Formula}, $\beta_{(n-d)+1,n}(S/I_\Delta) \neq 0$. So that $\projdim S/I_\Delta \geq (n-d)+1$. Hence by Theorem~\ref{betti, reg and C-M with indeg=1+dim}(iv), $\projdim S/I_\Delta= (n-d)+1$.

($\Rightarrow$) If $\projdim S/I_\Delta= (n-d)+1$, then $\beta_{(n-d)+1}(S/I_\Delta) \neq 0$. Hence there exists $1 \leq j \leq n$ such that $\beta_{(n-d)+1,j}(S/I_\Delta) \neq 0$. Using Theorem~\ref{betti, reg and C-M with indeg=1+dim}(i), $j \geq n$. Hence $j=n$. Thus,
\begin{equation*}
\begin{split}
0 \neq \beta_{(n-d)+1} \left( \frac{S}{I_\Delta} \right) & = \sum\limits_{j=1}^{n} \beta_{(n-d)+1,j}\left( \frac{S}{I_\Delta} \right)\\
& = \beta_{(n-d)+1,n}\left( \frac{S}{I_\Delta} \right) \\
& = \dim \tilde{H}_{d-2}(\Delta; K). \qquad \text{(By Theorem~\ref{Hochster Formula})}
\end{split}
\end{equation*}
\end{proof}

Now, let $\Delta$ be a $(d-1)$-dimensional simplicial complex on $[n]$ such that ${\rm indeg}\,(I_\Delta) =d$. As a consequence of Theorem~\ref{betti, reg and C-M with indeg=1+dim}, we conclude that:

\begin{cor} \label{linearity with indeg=1+dim}
Let $\Delta$ be a $(d-1)$-dimensional simplicial complex on $[n]$ such that ${\rm indeg}\,(I_\Delta) =d$. Then, $I = I_\Delta$ has a $d$-linear resolution if and only if $\tilde{H}_{d-1}(\Delta;K) = 0$.
\end{cor}

\begin{proof}
If $I$ has a $d$-linear resolution, then by Theorem~\ref{Hochster Formula}, we have:
\begin{equation*}
\begin{split}
0= \beta_{n-d-1,n} \left( I_\Delta \right) & = \dim_K \tilde{H}_{d-1}(\Delta;K). \\
\end{split}
\end{equation*}
Assume that $I$ does not have $d$-linear resolution, by Theorem~\ref{betti, reg and C-M with indeg=1+dim}(ii), we have:
\begin{equation*}
d+1 = \reg (I) = \max \{ j-i \colon \beta_{i,j} (I_\Delta) \neq 0 \}.
\end{equation*}
Let $d+1 = j_0-i_0$ and $\beta_{i_0j_0}(I_\Delta) \neq 0$. Then by Theorem~\ref{Hochster Formula}, there exists $W \subset [n]$ with $|W|=j_0$ and $\tilde{H}_{d-1}(\Delta_W; K) \neq 0$. This in particular implies that $\tilde{H}_{d-1}(\Delta; K) \neq 0$, for $\tilde{H}_{d-1}(\Delta_W; K) \subset \tilde{H}_{d-1}(\Delta; K)$.
\end{proof}

\section{Clutters and Clique Complexes}\label{Section4}
\begin{defn}[Clutter] \label{SC} \rm
A \textit{clutter} $\C$ on a vertex set $[n]$ is a set of subsets of $[n]$ (called \textit{circuits} of $\C$) such that if $e_1$ and $e_2$ are distinct circuits of $\C$ then $e_1 \nsubseteq e_2$.
A \textit{$d$-circuit} is a circuit consisting of exactly $d$ vertices, and a clutter is \textit{$d$-uniform} if every circuit has $d$ vertices.
A $(d-1)$-subset $e \subset [n]$ is called an \textit{submaximal circuit} of $\C$ if there exists $F \in \C$ such that $e \subset F$. The set of all submaximal circuit of $\C$ is denoted by ${\rm SC}(\C)$.
For $e \in {\rm SC}(\C)$, we denote by $\deg_\C(e)$, the \textit{degree} of $e$ to be
$$\deg_\C(e) = |\{F \in \C \colon \quad e \subset F \}|.$$
For a subset $W \subset [n]$, the \emph{induced subclutter} of $\C$ on $W$, $\C_W$ is a clutter with vertices $W$ and those circuits of $\C$ which their vertices are in $W$.
\end{defn}

For a non-empty clutter $\C$ on vertex set $[n]$, we define the ideal $I(\C)$, as follows:
$$I(\C) = \left(  \textbf{x}_T \colon \quad T \in \C \right),$$
where $\textbf{x}_T=x_{i_1}\cdots x_{i_t}$ for $T=\{i_1,\ldots,i_t\}$, and we define $I(\varnothing) = 0$.

Let $n\geq d$ be positive integers. We define $\C_{n,d}$, the \textit{maximal $d$-uniform clutter on $[n]$} as following:
$$\C_{n,d}=\{F \subset [n] \colon \quad |F|=d\}.$$
One can check that $I(\C_{n,d})$ has $d$-linear resolution (see also \cite[Example 2.12]{maar}).

If $\C$ is a $d$-uniform clutter on $[n]$, we define $\bar{\C}$, the \textit{complement} of $\C$, to be
$$\bar{\C}= \C_{n,d} \setminus \C = \{F \subset [n] \colon \quad |F|=d, \,F \notin \C\}.$$
Frequently in this paper, we take a $d$-uniform clutter $\C$ and we consider the square-free ideal $I=I(\bar{\C})$ in the polynomial ring $S=K[x_1, \ldots, x_n]$. We call $I= I(\bar{\C})$ the \textit{circuit ideal} of $\C$.

\begin{defn}[Clique Complex] \rm
Let $\C$ be a $d$-uniform clutter on $[n]$. A subset $V \subset [n]$ is called a \textit{clique} in $\C$, if all $d$-subsets of $V$ belongs to $\C$. Note that a subset of $[n]$ with less than $d$ elements is supposed to be a clique. The simplicial complex generated by cliques of $\C$ is called \textit{clique complex} of $\C$ and is denoted by $\Delta(\C)$.
\end{defn}

\begin{rem} \rm \label{initial degree of clique complex}
Let $\C$ be a $d$-uniform clutter on $[n]$ and $\Delta = \Delta(\C)$ be its clique complex. Then by our definition, all the subsets of $[n]$ with less than $d$ elements are also in $\Delta(\C)$. In particular, this implies that ${\rm indeg}\; I_{\Delta} \geq d$. So that by Proposition~\ref{Zero Reduced Homology of indeg=1+dim}, we have:
\begin{equation}
\tilde{H}_i(\Delta_W; K)=0, \qquad \text{for all }i<d-2 \text{ and } W \subset [n].
\end{equation}
\end{rem}

\begin{prop} \label{I_Delta = I (bar C)}
Let $\C$ be a $d$-uniform clutter on $[n]$ and $I=I(\bar{\C}) \subset K[x_1, \ldots, x_n]$ be the circuit ideal. Let $\Delta = \Delta(\C)$ be the clique complex of $\C$. Then,
\begin{itemize}
\item[\rm (i)] $\C = \mathcal{F}\left( \Delta^{(d-1)} \right)$;
\item[\rm (ii)] For all $u \in G(I_\Delta)$, $\deg (u) = d$;
\item[\rm (iii)] $I_\Delta = I$.
\end{itemize}
\end{prop}

\begin{proof}
We know that,
\begin{equation*}
I_\Delta = \bigcap\limits_{F \in \mathcal{F}(\Delta)} P_{\bar{F}}.
\end{equation*}
So that,
\begin{equation} \label{Local EQ 10}
\text{\rm \textbf{x}}_T \in I_\Delta \Longleftrightarrow T \cap \left( [n]\setminus F \right) \neq \varnothing, \text{ for all } F \in \mathcal{F}(\Delta).
\end{equation}

(i) Clear.

(ii) Let $u= \text{\rm \textbf{x}}_T \in G(I_\Delta)$. By Remark~\ref{initial degree of clique complex}, we know that $\deg (u) = |T| \geq d$.

If $\deg (u) = |T| > d$, then for all $d$-subset $T'$ of $T$, $\text{\rm \textbf{x}}_{T'} \notin I_\Delta$. This means that $T' \in \Delta$ for all $d$-subset $T'$ of $T$ (i.e. $T$ is a clique in $\C$). So that $T \in \Delta$ which is contradiction to the fact that $u= \text{\rm \textbf{x}}_T \in G(I_\Delta)$.

(iii) Let $T \in \bar{\C}$ and $\text{\rm \textbf{x}}_T \notin I_\Delta$. Then, by (\ref{Local EQ 10}), there exist $F \in \mathcal{F}(\Delta)$ such that $T \subset F$. Since $T$ is a $d$-subset of $F$, so $T \in \C$ which is contradiction. So that $I(\bar{\C}) \subset I_\Delta$.

For the converse, let $\text{\rm \textbf{x}}_T \in G(I_\Delta)$. Then, $T \notin \Delta$. Using part (i),  $T \notin \C$. Moreover, by (ii), we have $|T| = d$. Since $|T| = d$ and $T \notin \C$, one can say $T \in \bar{\C}$. This means that $I_\Delta \subset I(\bar{\C})$. This completes the proof.
\end{proof}

\begin{defn} \rm
A $d$-uniform clutter $\C$ is called \textit{decomposable} if there exists proper $d$-uniform subclutters $\C_1$ and $\C_2$ such that $\C= \C_1 \cup \C_2$ and either $V(\C_1) \cap V(\C_2)$ is a clique or ${\rm SC}(\C_1) \cap {\rm SC}(\C_2)= \varnothing$.

In this case, we write $\C = \C_1 \uplus \C_2$. A $d$-uniform clutter is said to be \textit{indecomposable} if it is not decomposable. For $d=2$, this definition coincides with the definition of decomposable graphs in \cite{HHBook}.
\end{defn}
Below we will find the regularity of the circuit ideal of $\C$ in terms of circuit ideals of $\C_1$ and $\C_2$, whenever $\C = \C_1 \uplus \C_2$. First we need the following lemma.

\begin{lem} \label{Nice lemma for union of clutters}
Let $\C_1$ and $\C_2$ be $d$-uniform clutters on two vertex sets $V_1$ and $V_2$ and $\C=\C_1 \cup \C_2$. Let $\Delta$ (res. $\Delta_1, \Delta_2$) be the clique complex of  $\C$ (res. $\C_1, \C_2$).
\begin{itemize}
\item[\rm (i)] If $G \subset V_1 \cup V_2$ and $G \cap (V_1 \setminus V_2) \neq \varnothing$, $G \cap (V_2 \setminus V_1) \neq \varnothing$, then $G \in \Delta \Longleftrightarrow |G| \leq d-1$.
\item[\rm (ii)] $\tilde{H}_i(\Delta; K) \cong \tilde{H}_i(\Delta_1 \cup \Delta_2; K)$, for all $i>d-2$.
\end{itemize}
\end{lem}

\begin{proof}
(i) Let $G$ be a subset of $V_1 \cup V_2$, as in (ii). If $|G| \leq d-1$, then by definition, $G$ is a clique in $\C$ and $G \in \Delta$.\\
Now, let $|G| \geq d$ and $x \in G \cap (V_1 \setminus V_2)$, $y \in G \cap (V_2 \setminus V_1)$. If $F$ be a $d$-subset of $G$ which contains $x,y$, then by Proposition~\ref{I_Delta = I (bar C)}(i), $F \notin \C_1 \cup \C_2= \C$. Hence $G \notin \Delta$.

(ii) First note that for $F \in \Delta$, we have:
\begin{equation} \label{Local EQ 00}
F \in \Delta_i \Longleftrightarrow F \subset V_i, \quad \text{for } i=1,2.
\end{equation}
Now, let
$$\Delta_3 = \langle G \in \Delta \colon \quad G \cap (V_1 \setminus V_2) \neq \varnothing, \;G \cap (V_2 \setminus V_1) \neq \varnothing \rangle.$$
Then (i) and (\ref{Local EQ 00}), imply that:
\begin{equation*}
\dim \Delta_3 = d-2, \qquad \Delta = \Delta_1 \cup \Delta_2 \cup \Delta_3.
\end{equation*}
It is clear that $\dim (\Delta_1 \cap \Delta_3) = \dim (\Delta_2 \cap \Delta_3) = d-3$. In particular,
\begin{align*}
\tilde{H}_i\left( (\Delta_1 \cup \Delta_2) \cap \Delta_3; K \right) =0, \quad \text{ for all } i>d-3.
\end{align*}
Hence from~(\ref{Reduced Mayer-Vietoris sequence}), for all $i>d-2$, we have:
\begin{align*}
\tilde{H}_i(\Delta; K) \cong \tilde{H}_i(\Delta_1 \cup \Delta_2; K) \oplus \tilde{H}_i(\Delta_3; K) = \tilde{H}_i(\Delta_1 \cup \Delta_2; K).
\end{align*}
\end{proof}

\begin{cor} \label{Homology with clique in intersection}
Let $\C= \C_1 \cup \C_2$ be a $d$-uniform clutter and $\Delta$ (res. $\Delta_1, \Delta_2$) be the clique complex of  $\C$ (res. $\C_1, \C_2$). If $V(\C_1) \cap V(\C_2)$ is a clique in $\C$, then:
$$\tilde{H}_i(\Delta; K) \cong \tilde{H}_i(\Delta_1; K) \oplus \tilde{H}_i(\Delta_2; K), \quad \text{ for all } i>d-2.$$
\end{cor}

\begin{proof}
By our assumption, $\Delta_1 \cap \Delta_2$ is a simplex. So that $\tilde{H}_i(\Delta_1 \cap \Delta_2; K)=0$ for all $i$. Using (\ref{Reduced Mayer-Vietoris sequence}), for all $i>0$, we have:
$$\tilde{H}_i(\Delta_1 \cup \Delta_2; K) \cong \tilde{H}_i(\Delta_1; K) \oplus  \tilde{H}_i(\Delta_2; K).$$
In addition to Lemma~\ref{Nice lemma for union of clutters}(ii), we get the conclusion.
\end{proof}

\begin{cor} \label{Homology with disjoint edge set}
Let $\C= \C_1 \cup \C_2$ be a $d$-uniform clutter and $\Delta$ (res. $\Delta_1, \Delta_2$) be the clique complex of  $\C$ (res. $\C_1, \C_2$). If ${\rm SC}(\C_1) \cap {\rm SC}(\C_2)= \varnothing$, then:
$$\tilde{H}_i(\Delta; K) \cong \tilde{H}_i(\Delta_1; K) \oplus \tilde{H}_i(\Delta_2; K), \quad \text{ for all } i>d-2.$$
\end{cor}

\begin{proof}
By our assumption, $\dim (\Delta_1 \cap \Delta_2) \leq d-2$. So that $\tilde{H}_i(\Delta_1 \cap \Delta_2; K)=0$ for all $i> d-2$. Using (\ref{Reduced Mayer-Vietoris sequence}), for all $i>d-1$, we have:
$$\tilde{H}_i(\Delta_1 \cup \Delta_2; K) \cong \tilde{H}_i(\Delta_1; K) \oplus  \tilde{H}_i(\Delta_2; K)$$
and $\tilde{H}_{d-1}(\Delta_1; K) \oplus  \tilde{H}_{d-1}(\Delta_2; K) \hookrightarrow \tilde{H}_{d-1}(\Delta_1 \cup \Delta_2; K).$\\

\noindent We claim that $\tilde{H}_{d-1}(\Delta_1; K) \oplus  \tilde{H}_{d-1}(\Delta_2; K) \cong \tilde{H}_{d-1}(\Delta_1 \cup \Delta_2; K).$
\begin{itemize}
\item[] \textbf{Proof of claim. }Let $\mathscr{C}(\Delta, \partial)$ (res. $\mathscr{C}(\Delta_1, \partial^{(1)})$, $\mathscr{C}(\Delta_2, \partial^{(2)})$) be the chain complex of $\Delta$ (res. $\Delta_1$, $\Delta_2$). Since ${\rm SC}(\C_1) \cap {\rm SC}(\C_2) = \varnothing$, we have:
\begin{equation} \label{Local EQ 15}
\bigoplus\limits_{F \in \Delta \atop \dim F=d-1} KF= \left( \bigoplus\limits_{F \in \Delta_1 \atop \dim F=d-1} KF \right) \oplus \left( \bigoplus\limits_{F \in \Delta_2 \atop \dim F=d-1} KF \right).
\end{equation}
Take $0 \neq F+ {\rm Im}\, \partial_{d} \in \tilde{H}_{d-1}(\Delta; K)$. Then by (\ref{Local EQ 15}), we can separate $F$ as $F= (c_1F_1 + \cdots + c_rF_r) + (c'_1G_1 + \cdots + c'_sG_s)$ where $c_i, c'_i \in K$ and $F_i \in \C_1, G_i \in \C_2$. Let
\begin{align*}
\partial_{d-1} (c_1F_1 + \cdots + c_rF_r)=(d_1e_1 + \cdots + d_{r'}e_{r'}) \\
\partial_{d-1} (c'_1G_1 + \cdots + c'_sG_s)= (d'_1f_1 + \cdots + d'_{s'}f_{s'})
\end{align*}
where, $d_i, d'_i \in K$ and $e_i \in {\rm SC}(\C_1), f_i \in {\rm SC}(\C_2)$. Since
\begin{align*}
0=\partial_d (F)& =\partial_{d-1} (c_1F_1 + \cdots + c_rF_r)+ \partial_{d-1} (c'_1G_1 + \cdots + c'_sG_s)\\
&=(d_1e_1 + \cdots + d_{r'}e_{r'}) + (d'_1f_1 + \cdots + d'_{s'}f_{s'})
\end{align*}
and ${\rm SC}(\C_1) \cap {\rm SC}(\C_2) = \varnothing$, we conclude that
$$\partial_{d-1} (c_1F_1 + \cdots + c_rF_r) = \partial_{d-1} (c'_1G_1 + \cdots + c'_sG_s)=0.$$
This means that the natural map
$$\tilde{H}_{d-1}(\Delta_1; K) \oplus  \tilde{H}_{d-1}(\Delta_2; K) \hookrightarrow \tilde{H}_{d-1}(\Delta_1 \cup \Delta_2; K)$$
is onto too.
\end{itemize}
By what we have already proved, we have:
\begin{equation*}
\tilde{H}_{i}(\Delta_1; K) \oplus  \tilde{H}_{i}(\Delta_2; K) \cong \tilde{H}_{i}(\Delta_1 \cup \Delta_2; K), \quad \text{for all } i>d-2.
\end{equation*}
In addition with Lemma~\ref{Nice lemma for union of clutters}(ii), we get the conclusion.
\end{proof}

\begin{rem} \rm
Let $\C_1, \C_2$ be $d$-uniform clutters on vertex set $V_1, V_2$ with $V_1 \cup V_2=[n]$ and $\C=\C_1 \cup \C_2$. For all $W \subset [n]$, one can easily check that:
\begin{itemize}
\item[\rm (i)] $\C_W = ({\C_1})_W \cup ({\C_2})_W$.
\item[\rm (ii)] $\Delta_W = \Delta(\C_W)$.
\item[\rm (iii)] ${\rm SC}(({\C_1})_W) \cap {\rm SC}(({\C_2})_W) = ({{\rm SC}(\C_1 \cap \C_2)})_W$.
\end{itemize}
Hence, if $V_1 \cap V_2$ is a clique or ${\rm SC}(\C_1) \cap {\rm SC}(\C_2) = \varnothing$, then (i)-(iii) and Corollary \ref{Homology with clique in intersection}, \ref{Homology with disjoint edge set}, imply that:
\begin{equation} \label{Local EQ 16}
\tilde{H}_i(\Delta_W;K) \cong \tilde{H}_i({(\Delta_1)}_W;K) \oplus \tilde{H}_i({(\Delta_2)}_W;K), \quad \text{for all }i>d-2.
\end{equation}
\end{rem}

Now we present the main theorem of this section.

\begin{thm} \label{betti number with E(C_1) cap E(C_2)= empty}
Let $\C = \C_1 \uplus \C_2$ be a $d$-uniform clutter and $I$ (res. $I_1, I_2$) be the circuit ideals of $\C$ (res. $\C_1, \C_2$). Then,
\begin{itemize}
\item[\rm (i)] $\beta_{i,j}(I) \geq \beta_{i,j}(I_1)+\beta_{i,j}(I_2)$, for $j-i>d$.
\item[\rm (ii)] If $I_1$ and $I_2$ are non-zero ideals, then $\reg (I) = \max \{\reg (I_1), \;\reg (I_2) \}$.
\end{itemize}
\end{thm}

\begin{proof}
(i) Let $\Delta$ (res. $\Delta_1, \Delta_2$) be the clique complex of  $\C$ (res. $\C_1, \C_2$). Then, by (\ref{Local EQ 16}) and Theorem~\ref{Hochster Formula}, for $j-i>d$, we have:
\begin{align*}
\beta_{i,j}(I_\Delta) & = \sum\limits_{W \subset [n] \atop |W|=j} \dim_K \tilde{H}_{j-i-2}(\Delta_W;K) \\
&= \sum\limits_{W \subset [n] \atop |W|=j} \left[ \dim_K \tilde{H}_{j-i-2}({(\Delta_1)}_W;K) + \dim_K \tilde{H}_{j-i-2}({(\Delta_2)}_W;K) \right] \\
& =  \sum\limits_{W \subset [n] \atop |W|=j} \dim_K \tilde{H}_{j-i-2}({(\Delta_1)}_W;K) +  \sum\limits_{W \subset [n] \atop |W|=j} \dim_K \tilde{H}_{j-i-2}({(\Delta_2)}_W;K) \\
& \geq \beta_{i,j}(I_{\Delta_1})+ \beta_{i,j}(I_{\Delta_2}).
\end{align*}
Hence by Proposition~\ref{I_Delta = I (bar C)}(iii), $\beta_{i,j}(I) \geq \beta_{i,j}(I_1)+\beta_{i,j}(I_2)$, whenever $j-i>d$.

(ii) If $I$ has a $d$-linear resolution, $\beta_{i,j}(I) =0$ for all $j-i>d$. So that (i) implies that $\beta_{i,j}(I_1) = \beta_{i,j}(I_1) =0$, for all $j-i>d$. This means that, both of ideals $I_1$ and $I_2$ have a $d$-linear resolution and the equality $\reg (I) = \max \{\reg (I_1), \reg (I_2) \}$ holds.

Assume that, $I$ does not have $d$-linear resolution. Let
$$r=\reg(I) = \max\{ j-i \colon \quad \beta_{i,j}(I) \neq 0 \}$$
and $j_0, i_0$ be such that $r =j_0 - i_0$ with $\beta_{i_0,j_0}(I) \neq 0.$ By Theorem~\ref{Hochster Formula}, there exists a $W \subset [n]$, with $|W|=j_0$ and $\tilde{H}_{r-2}(\Delta_W; K) \neq 0$. Since $r-2>d-2$, from (\ref{Local EQ 16}), we conclude that, either 
\begin{align*}
\tilde{H}_{r-2}({(\Delta_1)}_W;K) \neq 0 \quad \text{or} \quad \tilde{H}_{r-2}({(\Delta_2)}_W;K) \neq 0.
\end{align*}
Without loss of generality, we may assume that $\tilde{H}_{r-2}({(\Delta_1)}_W;K) \neq 0$ and we put $W'=W \cap V(\Delta_1)$. Then, $W'$ is a subset of the vertex set of $\Delta_1$ with the property that $\tilde{H}_{r-2}({(\Delta_1)}_{W'}; K) \neq 0$. Using Theorem~\ref{Hochster Formula} once again, we have:
\begin{align*}
\beta_{|W'|-r, |W'|}(I_1) = \sum\limits_{T \subset V(\Delta_1) \atop |T| = |W'|} \dim_K \tilde{H}_{r-2}\left( (\Delta_1)_T; K \right) \geq \dim_K \tilde{H}_{r-2}\left( (\Delta_1)_{W'}; K \right) > 0.
\end{align*}
Hence, $\beta_{|W'|-r, |W'|}(I_1) \neq 0$ and,
\begin{align*}
\max \{\reg (I_1), \reg (I_2) \} \geq \reg(I_1) & = \max\{ j-i \colon \; \beta_{i,j}(I_1) \neq 0 \} \\
& \geq (|W'|)- (|W'|-r) = r.
\end{align*}
The inequality, $\max \{\reg (I_1), \reg (I_2) \} \leq r$ comes from (i). Putting together these inequalities, we get the conclusion.

\end{proof}

The following example shows that, the inequality $\beta_{i,j}(I) \geq \beta_{i,j}(I_1)+\beta_{i,j}(I_2)$, for $j-i>d$ in Theorem~\ref{betti number with E(C_1) cap E(C_2)= empty}, may be strict.
\begin{ex} \rm
Consider the $3$-uniform clutter $\C = \{ 123, 124, 134, 235, 245, 345, 347, 367, 467, 356, 456 \}.$

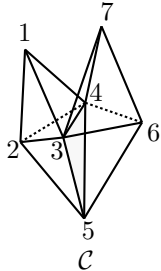
\begin{center}
\psset{xunit=0.77cm,yunit=0.77cm,algebraic=true,dotstyle=o,dotsize=3pt 0,linewidth=0.8pt,arrowsize=3pt 2,arrowinset=0.25}
\begin{pspicture*}(2.86,0.0)(5.78,4.82)
\pspolygon[linestyle=none,fillstyle=solid,fillcolor=lightgray,opacity=0.1](4.38,2.96)(4,2.36)(4.36,0.96)
\psline(4.38,2.96)(4,2.36)
\psline(4,2.36)(4.36,0.96)
\psline(4.36,0.96)(4.38,2.96)
\psline(3.34,3.88)(4.38,2.96)
\psline(3.34,3.88)(4,2.36)
\psline(3.34,3.88)(3.26,2.28)
\psline(3.26,2.28)(4.36,0.96)
\psline(3.26,2.28)(4,2.36)
\psline[linestyle=dashed,dash=1pt 1pt](3.26,2.28)(4.38,2.96)
\psline(4.66,4.28)(4.38,2.96)
\psline(4.66,4.28)(4,2.36)
\psline(4.66,4.28)(5.36,2.62)
\psline(4,2.36)(5.36,2.62)
\psline[linestyle=dashed,dash=1pt 1pt](4.38,2.96)(5.36,2.62)
\psline(5.36,2.62)(4.36,0.96)
\rput[tl](3.22,4.32){$1$}
\rput[tl](3.02,2.3){$2$}
\rput[tl](3.78,2.28){$3$}
\rput[tl](4.45,3.28){$4$}
\rput[tl](4.32,0.9){$5$}
\rput[tl](5.45,2.6){$6$}
\rput[tl](4.64,4.68){$7$}
\rput[tl](4.26,0.4){$\C$}
\end{pspicture*}
\end{center}

Let $\C_1 = \{123, 124, 134, 235, 245, 345 \}$ and $\C_2 = \{ 345, 347, 367, 467, 356, 456\}$. Then, $\C = \C_1 \uplus \C_2$ and a direct computation using \cocoa, shows that the minimal free resolution of the ideal $I(\bar{\C})$ is:
\begin{align*}
0 \to S^6(-7) \to S^{30}(-6) \oplus S^2(-7) \to S^{62}(-5) \oplus S^4(-6) \to S^{61}(-4) \oplus S^2(-5) \to S^{24}(-3) \to I \to 0.
\end{align*}
Note that $\beta^K_{2,6}(I(\bar{\C}_1)) = \beta^K_{2,6}(I(\bar{\C}_2)) = 0$, while $\beta^K_{2,6}(I(\bar{\C})) =4$.
\end{ex}

\begin{rem} \label{Remark on regularity of union of clutters} \rm
Let $\C = \C_1 \uplus \C_2$ be a $d$-uniform clutter on $[n]$ and $I$ (res. $I_1, I_2$) be the circuit ideals of $\C$ (res. $\C_1, \C_2$). Let $\Delta$ (res. $\Delta_1, \Delta_2$) be the clique complex of  $\C$ (res. $\C_1, \C_2$).
\begin{itemize}
\item If both of $I_1$ and $I_2$ are zero ideals, then $\Delta_1$ and $\Delta_2$ are simplexes and they have zero reduced homologies in all degrees. So that $\tilde{H}_i(\Delta_W; K) = 0$ for all $W \subset [n]$ and $i>d-2$ by (\ref{Local EQ 16}). So that $\beta_{i,j}(I) =0$ for all $j-i>d$. That is, the ideal $I$ has a $d$-linear resolution.

\item If only one of the ideals $I_1$ or $I_2$ is a zero ideal, say $I_1$, then $\Delta_1$ is a simplex and all the reduced homologies of $\Delta_1$ is zero. Using (\ref{Local EQ 16}), we conclude that $\tilde{H}_i(\Delta_W;K) \cong \tilde{H}_i({(\Delta_2)}_W;K)$ for all $W \subset [n]$ and $i>d-2$. This implies that $\reg (I) = \reg(I_2)$.

\item If $I_1$ and $I_2$ are non-zero ideals, then Theorem~\ref{betti number with E(C_1) cap E(C_2)= empty}(ii) implies that 
$$\reg (I) = \max \{\reg (I_1), \reg (I_2) \}.$$
\end{itemize}
\end{rem}

\section{Minimal to $d$-linearity}\label{Section5}
In this section, we define three classes of clutters which their circuit ideals do not have $d$-linear resolution but the circuit ideal of any proper subclutter of them has a $d$-linear resolution.

A clutter $\C$ is said to be \emph{connected} if for each two vertices $v_1$ and $v_2$, there is a sequence of circuits $F_1,\ldots,F_r$ such that $v_1\in F_1, v_2\in F_r$ and $F_i\cap F_{i+1}\neq \varnothing$. A connected  $d$-uniform clutter $\C$ is called a \emph{tree} if any subclutter of $\C$ has a submaximal circuit of degree one. A union of some trees is called a \emph{forest}. By Remark~3.10 of \cite{maar}, the circuit ideal of any $d$-uniform forest has a $d$-linear resolution.

\begin{defn} \label{Definition of minimal to linearity} \rm
Let $\C$ be a $d$-uniform clutter on $[n]$, $\Delta= \Delta(\C)$ its clique complex. Suppose that $I= I(\bar{\C}) \subset K[x_1, \ldots, x_n]$, the circuit ideal of $\C$, does not have $d$-linear resolution.
\begin{itemize}
\item[\rm (i)] The clutter $\C$ is called \textit{obstruction to $d$-linearity} if for every proper subclutter $\C' \varsubsetneq \C$, the ideal $I(\bar{\C}')$ has a $d$-linear resolution.
\item[\rm (ii)] The clutter $\C$ is called \textit{minimal to $d$-linearity} if it is obstruction to $d$-linearity and $\dim \Delta = d-1$.
\item[\rm (iii)] The clutter $\C$ is called \textit{almost tree} if every proper subclutter of $\C$ is a tree.
\end{itemize}
Let $\mathscr{C}^{\rm obs}_d$, $\mathscr{C}^{\rm min}_d$ and $\mathscr{C}^{\rm a.tree}_d$  denote the classes of clutters which are obstruction to $d$-linearity, minimal to $d$-linearity and almost tree, respectively.

Note that if $\C \in \mathscr{C}^{\min}_d$ and $\Delta=\Delta(\C)$ is its clique complex, then we have:
\begin{equation}
{\rm indeg}\, I_\Delta = {\rm indeg}\, I(\bar{\C}) = d = 1+\dim \Delta .
\end{equation}
\end{defn}

\begin{lem} \label{Properties of minimal to linearity}
Let $\C$ be a $d$-uniform clutter on $[n]$ which is minimal to $d$-linearity and $\Delta= \Delta(\C)$ be the clique complex of $\C$. Then,
\begin{itemize}
\item[\rm (i)] $\dim_K \tilde{H}_{d-1}(\Delta; K)=1$.
\item[\rm (ii)] If $W \varsubsetneq [n]$, then $\tilde{H}_{d-1}(\Delta_W; K)=0$.
\end{itemize}
\end{lem}

\begin{proof}

(i) Let $0 \neq F= c_1F_1+ \cdots + c_rF_r \in \tilde{H}_{d-1}(\Delta; K)$ where $c_i \in K$ and $F_i \in \C$. Then, ${\rm Supp}(F) := \{F_i \colon \quad c_i \neq 0\}$ is equal to $\C$, because every proper subclutter of $\C$ has linear resolution.

If $\dim_K \tilde{H}_{d-1}(\Delta; K)>1$ and $F= c_1F_1+ \cdots + c_rF_r, G= d_1F_1+ \cdots + d_rF_r$ be two basis element of $\tilde{H}_{d-1}(\Delta; K)$, then $0 \neq c_1G-d_1F \in \tilde{H}_{d-1}(\Delta; K)$ and $
{\rm Supp}(c_1G-d_1F) \varsubsetneq \C$ which is a contradiction.

(ii) One can easily check that $\Delta_W = \Delta(\C_W)$ for all $W \subset [n]$. By definition, for all $W \varsubsetneq [n]$, the induced clutter $\C_W$ has linear resolution. So that by Theorem~\ref{Hochster Formula}, $\tilde{H}_{d-1}(\Delta_W; K)= \tilde{H}_{d-1}(\Delta(\C_W); K)= 0$.
\end{proof}

The following is the main theorem of this section which gives an  explicit minimal free resolution for the circuit ideal of a clutter which is minimal to $d$-linearity. 

\begin{thm}\label{Resolution of Minimal to linearity}
Let $\C$ be a $d$-uniform clutter on $[n]$ which is minimal to $d$-linearity and $I= I(\bar{\C}) \subset K[x_1, \ldots, x_n]$ be the circuit ideal. Then the minimal free resolution of $I$ is
\begin{align}
0  \to S^{\beta_{n-d,n}}(-n) & \to S(-n) \oplus S^{\beta_{n-d-1, n-1}}(-(n-1)) \to  S^{\beta_{n-d-2, n-2}}(-(n-2)) \nonumber \\
& \to \cdots \to S^{\beta_{1,d+1}}(-(d+1)) \to S^{\beta_{0,d}}(-d) \to I \to 0 \label{Resolution of Minimal to linearity+Shape}
\end{align}
where,
\begin{itemize}
\item[\rm (i)] $\beta_{n-d,n}(I) = 1 -e(S/I) + \sum\limits_{i=0}^{d-1} (-1)^{d+i-1} {n \choose i}$.
\item[\rm (ii)] $ \beta_{i, i+d}(I) = {n-d \choose i} \left( \frac{d}{d+i}{n \choose d} - e(S/I) \right)$, for $0 \leq i \leq n-d-1$.
\end{itemize}
and $e(S/I) = {n \choose d} - \mu (I)$.
\end{thm}

\begin{proof}
Let $\Delta=\Delta(\C)$ be the clique complex of $\C$. Since ${\rm indeg}\, (I_\Delta) ={\rm indeg}\, I(\bar{\C})= d =1+\dim \Delta$, by Theorem~\ref{betti, reg and C-M with indeg=1+dim}(i) and Lemma~\ref{Properties of minimal to linearity}(ii), $\beta_{i,j}(I) =0$ either $j-i < d$ or $j-i > d+1$ or $j-i=d+1$ and $j<n$. Moreover, we have $\beta_{n-(d+1),n} = \dim_K \tilde{H}_{d-1}(\Delta; K)=1$. Hence the minimal free resolution of $I$ is in the form (\ref{Resolution of Minimal to linearity+Shape}). The equation (ii) comes from Theorem~\ref{dif-herzog-kuhl}. Using Theorem~\ref{Hochster Formula} once again, we have $\beta_{n-d,n}(I) = \dim_K \tilde{H}_{d-2}(\Delta; K)$. Hence (i) comes from Corollary~\ref{Homology of ideals with indeg I = 1+dim}. In order to find the multiplicity, note that $e(S/I) = f_{d-1}(\Delta) = |\C| = {n \choose d} -  \mu (I)$.
\end{proof}

Let $\C$ be a $d$-uniform clutter. The clutter $\C$ is called \textit{strongly connected} (or \textit{connected in codimension one}) if for any two circuits $F,G \in \C$, there exists a chain of circuits $F=F_0, \ldots ,F_s=G$ in $\C$ such that $|F_i \cap F_{i+1}|=d-1$, for $i=0, \ldots, s-1.$

Besides the algebraic properties of the clutters $\C \in \mathscr{C}^{\rm obs}_d$, a combinatorial property of such clutters is that they are strongly connected.

\begin{prop} \label{Strongly connectedness of obstructoin to linearity}
If $\C \in \mathscr{C}^{\rm obs}_d$ be a $d$-uniform clutter, then
\begin{itemize}
\item[\rm (i)] $\C$ is indecomposable.
\item[\rm (ii)] $\C$ is strongly connected.
\end{itemize}

\end{prop}

\begin{proof}
Let $\C= \C_1 \uplus \C_2$ where $\C_1$ and $\C_2$ are proper subclutters of $\C$. By definition, the ideals $I_1=I(\bar{\C}_1)$ and $I_2=I(\bar{\C}_2)$ have $d$-linear resolutions. In view of Remark~\ref{Remark on regularity of union of clutters}, the ideal $I(\bar{\C})$ has $d$-linear resolution which is a contradiction.

(ii) Let $\C_1 \subset \C$ be the maximal subclutter (w.r.t. inclusion) of $\C$ which is strongly connected. Clearly, $\C_1 \neq \varnothing$, because every clutter with one circuit is strongly connected.

Assume that $\C_1 \varsubsetneq \C$ and let $\C_2 = \C \setminus \C_1$. By the maximality of $\C_1$, ${\rm SC}(\C_1) \cap {\rm SC}(\C_2) = \varnothing$, that is $\C= \C_1 \uplus \C_2$ which contradicts to (i). So that $\C_1= \C$ is strongly connected.
\end{proof}



\begin{lem} \label{C(iso,d) is subset of C(min, d)}
Let $\C$ be a $d$-uniform clutter which is a tree or almost tree and $\Delta = \Delta(\C)$ be the clique complex of $\C$. Then, $\dim \Delta =d-1$. In particular, $\mathscr{C}^{\rm a.tree}_d \subset \mathscr{C}^{\min}_d$.
\end{lem}

\begin{proof}
If $G \in \Delta$ and $|G|>d$ and $V$ is the vertex set of $G$, then $\C_V=\{F \in \C \colon F \subset G \}$. Hence for all $e \in {\rm SC}(\C_V)$, $\deg_{\C_V}(e) \geq 2$. This contradicts to the fact that $\C_V$ has submaximal circuit of degree 1. So that all faces of $\Delta(\C)$ have at most $d$ elements. Since $\C \subset \Delta$, we conclude that $\dim \Delta = d-1$.

If $\C \in \mathscr{C}^{\rm a.tree}_d$, then by what we have already proved, we know that $\dim \Delta(\C) = d-1$. Also, the argument before Definition~\ref{Definition of minimal to linearity} implies that for every proper subclutter $\C' \varsubsetneq \C$, the ideal $I(\bar{\C}')$ has a linear resolution. Hence $\C \in \mathscr{C}^{\min}_d$.
\end{proof}

We have shown that $\mathscr{C}^{\rm a.tree}_d\subseteq\mathscr{C}^{\min}_d\subseteq\mathscr{C}^{\rm obs}_d$. All our evidences and computations lead us to make the following conjecture.

\begin{conj}
$\mathscr{C}^{\rm a.tree}_d=\mathscr{C}^{\min}_d=\mathscr{C}^{\rm obs}_d.$
\end{conj}

\section{Some Applications}\label{Section6}
{\it Fr\"{o}berg's Theorem}

Let $G$ be a simple graph ($2$-uniform clutter). Fr\"{o}berg \cite{Fr}, has proved that the ideal $I(\bar{G})$ has $2$-linear resolution if and only if $G$ is a chordal graph. A graph is called \textit{chordal} if each cycle in $G$ has a chord i.e. any minimal induced cycle in $G$ is  of length 3. In this section, we will present an alternative proof for this theorem.

Let $C_n$ be a cycle of length $n>3$. Though that the Betti numbers of the circuit ideal of $C_n$ is well-known, we can recover them using results of this paper.

Let $\Delta=\Delta(C_n)$ be the clique complex of $C_n$ and $I=I(\bar{C}_n)$ be the circuit ideal. Then ${\rm indeg}\, I_\Delta = 1+\dim \Delta$ and by Corollary~\ref{Homology of ideals with indeg I = 1+dim}, $\dim \tilde{H}_1(\Delta; K) =1$. In particular, $I$ does not have linear resolution (Corollary~\ref{linearity with indeg=1+dim}) and $C_n$ is minimal to $2$-linearity (Lemma~\ref{C(iso,d) is subset of C(min, d)}). Moreover, By Theorem~\ref{Resolution of Minimal to linearity}, the minimal free resolution of $I$ is
\begin{align*}
0 \to S(-n) \to S^{\beta_{n-4,n-2}}(-(n-2)) \to \cdots \to S^{\beta_{1,3}}(-3) \to S^{\beta_{0,2}}(-2) \to I \to 0
\end{align*}
where $\beta_{i, i+2}(I)= n {n-2 \choose i} \left( \frac{n-3-i}{2+i} \right)$ for $0 \leq i \leq n-4$.

So that, if a graph $G$ has a cycle as an induced subgraph, then by Theorem~\ref{Hochster Formula}, the ideal $I(\bar{G})$ does not have linear resolution. This means that, the ideal $I(\bar{G})$ does not have linear resolution if $G$ is not chordal.

Conversely, if $G \neq \C_{n,2}$ is chordal, then by Dirac Theorem \cite{Dirac} (see also \cite[Lemma 9.2.1]{HHBook}), there exists proper induced subgraphs $G_1$ and $G_2$ such that $G = G_1 \uplus G_2$. Since $G_1$ and $G_2$ are induced subgraphs of a chordal graph $G$, we conclude that $G_1$ and $G_2$ are chordal. Hence induction and Remark~\ref{Remark on regularity of union of clutters}, implies that the ideal $I(\bar{G})$ has a $2$-linear resolution.

\medskip
\noindent{\it Generalized Chordal Clutters}

E. Emtander \cite{Emtander} has defined generalized chordal clutters as the following.

\begin{defn} \label{Emtander chordal} \rm
A \textit{generalized chordal clutter} is a $d$-uniform clutter, obtained inductively as follows:
\begin{itemize}
\item[\rm (a)] $\C_{n,d}$ is a generalized chordal clutter.
\item[\rm (b)] If $\mathcal{G}$ is generalized chordal clutter, then so is $\C = \mathcal{G} \cup_{\C_{i,d}} {\C_{n,d}}$ for all $0 \leq i <n$.
\item[\rm (c)] If $\mathcal{G}$ is generalized chordal and $V \subset V(\mathcal{G})$ is a finite set with $|V| = d$ and at least one element of $\{F \subset V: |F|=d-1\}$ is not a subset of any element of $\mathcal{G}$, then $\mathcal{G} \cup V$ is generalized chordal.
\end{itemize}
\end{defn}

Emtander has proved that the circuit ideal of generalized chordal clutters have $d$-linear resolution over any field $K$ (c.f. \cite[Theorem 5.1]{Emtander}). We can recover this result as an special case of Theorem~\ref{betti number with E(C_1) cap E(C_2)= empty}.

Let $\C$ be a generalized chordal clutter. If $\C$ has a circuit $F$, with property (c) in the above definition, then Remark 3.10 of \cite{maar} together with induction, implies that $I(\bar{\C})$ has a $d$-linear resolution. So we may assume that $\C= \mathcal{G} \cup_{\C_{i,d}} \C_{n,d}$. Again, in this case, Remark~\ref{Remark on regularity of union of clutters} together with induction, implies that the ideal $I(\bar{\C})$ has a $d$-linear resolution over the field $K$.

\medskip
\noindent{\it Resolution of Pseudo-Manifolds}
\vspace{-0.15cm}
\begin{defn} \rm
A $d$-uniform clutter $\C$ is called a \textit{pseudo-manifold}, if $\C$ is strongly connected and each $e \in {\rm SC}(\C)$ has degree 2.
\end{defn}

For more details on pseudo-manifolds and the concept of orientability, refer to \cite{Massey} Chapter IX.

\begin{lem} \label{proper subclutter of Pseudo-manifolds}
Let $\C$ be a $d$-uniform clutter such that $\deg_\C(e)=2$ for all $e \in {\rm SC}(\C)$. Then, every proper subclutter of $\C$ has a submaximal circuit of degree 1 if and only if $\C$ is strongly connected. In particular, every proper subclutter of a pseudo-manifold  is a tree.
\end{lem}

\begin{proof}
($\Rightarrow$) Let $F \in \C$ and $\C_1$ be a maximal subclutter of $\C$ which consists of all $G \in \C$ such that there exists a chain $F=F_0, F_1, \ldots, F_r=G$ of circuits of $\C$ such that $|F_i \cap F_{i+1}|=d-1$ for $i=0, \ldots, r-1$.

If $\C_1 \varsubsetneq \C$, then $\C_1$ has a submaximal circuit $e$ of degree 1. By the maximality of $\C_1$, we have:
$$1 = \deg_{\C_1}(e) = \deg_\C(e).$$
This contradicts to our assumption on $\C$.

($\Leftarrow$) Let $\C' \varsubsetneq \C$ such that $\deg_{\C'}(e) =2 =\deg_{\C}(e)$ for all $e \in {\rm SC}(\C')$. Take $F \in \C'$ and $G \in \C \setminus \C'$. By our assumption, there exist a chain $F=F_0, F_1, \ldots, F_r=G$ of circuits of $\C$ such that $|F_i \cap F_{i+1}|=d-1$ for $i=0, \ldots, r-1$.

Since $F_0 = F \in \C'$ and $|F_0 \cap F_1|=d-1$, we conclude that $F_0 \cap F_1 \in {\rm SC}(\C')$. Hence, by our assumption, $\deg_{\C'} (F_0 \cap F_1) =2$ which implies that $F_1 \in \C'$. The same argument shows that $F_0, F_1, \ldots, F_r$ are in $\C'$. This is a contradiction by our choice of $F_r=G$.
\end{proof}

\begin{rem} \label{Remark for Pseudo-manifolds are minimal to linearity}\rm
Let $\C$ be a $d$-uniform pseudo-manifold and $\Delta=\Delta(\C)$ be the clique complex of $\C$. In view of Lemmas~\ref{proper subclutter of Pseudo-manifolds} and \ref{C(iso,d) is subset of C(min, d)}, we have:
\begin{itemize}
\item[\rm (a)] Every proper subclutter of $\C$ has a submaximal circuit of degree 1.
\item[\rm (b)] ${\rm indeg}\, (I_\Delta) = 1+ \dim \Delta$.
\end{itemize}
Putting together these results, Corollary~\ref{linearity with indeg=1+dim} implies that:
\begin{center}
$I(\bar{\C})$ is minimal to $d$-linearity if and only if  $\tilde{H}_{d-1}(\Delta; K) \neq 0$.
\end{center}
\end{rem}

\begin{prop} \label{Pseudo-manifolds are minimal to linearity}
Let $\C$ be a $d$-uniform clutter.
\begin{itemize}
\item[\rm (i)] If $\C$ is oriented pseudo-manifold, then $\C$ is minimal to $d$-linearity.
\item[\rm (ii)] If $\C$ is non-oriented pseudo-manifold, then $\C$ is minimal to $d$-linearity if and only if ${\rm Char}(K)=2$.
\end{itemize}
\end{prop}

\begin{proof}
Let $\C$ be a $d$-uniform pseudo-manifold and $\Delta= \Delta(\C)$ be its clique complex. In view of Lemma~\ref{C(iso,d) is subset of C(min, d)}, we know that $\dim \Delta=d-1$ and $\C= \mathcal{F}(\Delta)$. In particular, $\tilde{H}_{d-1}(\Delta; K) \cong \tilde{H}_{d-1}(\langle \C \rangle; K)$ where $\langle \C \rangle$ is the simplicial complex generated by $\C$. But we know that (see \cite[Chapter X, Exercise 6.5] {Massey} or \cite[\S43, Exercise 5]{Munkres}):
\begin{equation*}
\tilde{H}_{d-1}(\langle \C \rangle; K) = \begin{cases}
K, & \text{if $\C$ is oriented.} \\
{\rm Tor}(\ZZ_2, K), & \text{if $\C$ is non-oriented.}
\end{cases}
\end{equation*}
where ${\rm Tor}(\ZZ_2, K)= \{ a \in K \colon \quad 2.a =0 \}$. Now, the conclusion follows from Remark~\ref{Remark for Pseudo-manifolds are minimal to linearity}.
\end{proof}

Note that if $\Delta$ is a triangulation of a connected compact $d$-manifold (or homology $d$-manifold), then $\C= \mathcal{F}(\Delta)$ is a $d$-uniform Pseudo-manifold (see \cite[\S43, \S63]{Munkres}). So that we may use Theorem~\ref{Resolution of Minimal to linearity} to find the minimal free resolution of the ideal $I(\bar{\C})$. It is worth to say that pseudo-manifolds are strictly contained in $\mathscr{C}^{\rm a.tree}_d$.

\begin{ex} \label{non-pseudo-manifold} \rm
Let $\Delta$ be a triangulation of the following shape and $\C= \mathcal{F} (\Delta)$. That is:
\begin{align*}
\Delta = \langle & a23, b14, ab1, a12, ab4, a34, 236, 367, 125, 256, 145, 458, 348, 378, \\
& a67, b58, ab5, a56, ab8, a78 \rangle.
\end{align*}
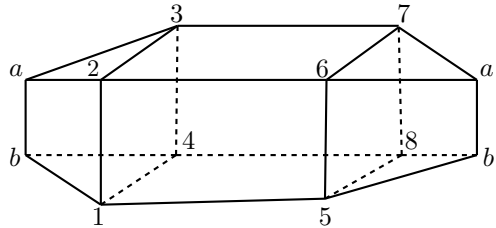
\begin{center}
\psset{xunit=1.0cm,yunit=1.0cm,algebraic=true,dotstyle=o,dotsize=3pt 0,linewidth=0.8pt,arrowsize=3pt 2,arrowinset=0.25}
\begin{pspicture*}(-0.52,1.0)(8.0,4.0)
\psline(1,3)(1,2)
\psline(2,1.34)(2,3)
\psline(2,3)(3.02,3.72)
\psline[linestyle=dashed,dash=2pt 2pt](3.02,3.72)(3,2)
\psline[linestyle=dashed,dash=2pt 2pt](3,2)(2,1.34)
\psline(3.02,3.72)(1,3)
\psline(2,3)(1,3)
\psline(1,2)(2,1.34)
\psline[linestyle=dashed,dash=2pt 2pt](1,2)(3,2)
\psline(3.02,3.72)(5.96,3.72)
\psline(2,3)(5,3)
\psline(2,1.34)(4.98,1.42)
\psline[linestyle=dashed,dash=2pt 2pt](3,2)(6,2)
\psline(5.96,3.7)(5,3)
\psline(5,3)(4.98,1.42)
\psline[linestyle=dashed,dash=2pt 2pt](5.96,3.7)(6,2)
\psline[linestyle=dashed,dash=2pt 2pt](6,2)(4.98,1.42)
\psline(7,2)(7,3)
\psline(7,2)(4.98,1.42)
\psline[linestyle=dashed,dash=2pt 2pt](7,2)(6,2)
\psline(7,3)(5,3)
\psline(7,3)(5.96,3.7)
\rput[tl](0.78,3.2){$a$}
\rput[tl](0.78,2.1){$b$}
\rput[tl](7.04,3.2){$a$}
\rput[tl](7.08,2.1){$b$}
\rput[tl](1.88,1.3){$1$}
\rput[tl](1.82,3.26){\small $2$}
\rput[tl](2.92,4.0){$3$}
\rput[tl](3.08,2.3){$4$}
\rput[tl](4.9,1.3){$5$}
\rput[tl](4.86,3.26){$6$}
\rput[tl](5.94,4.0){$7$}
\rput[tl](6.04,2.3){$8$}
\end{pspicture*}
\end{center}
Then, $\C$ is not a pseudo-manifold, because $\deg_\C (ab) = 4$, but $\C$ is almost tree and hence minimal to linearity.
\end{ex}

\begin{ex} \rm
Let $\Delta_1$ be a triangulation of a Torus and $\Delta_2$ be a triangulation of a projective plane such that they intersect in one triangle and let $\C = \mathcal{F}(\Delta_1) \cup \mathcal{F}(\Delta_2)$ be the corresponding $3$-uniform clutter on the vertex set $[n]$.

In view of Theorem~\ref{betti number with E(C_1) cap E(C_2)= empty}(ii), $\reg (I) =4$ in any characteristic of the base field.

\end{ex}


\end{document}